\documentclass[11pt,reqno]{amsart}
\usepackage[latin1]{inputenc}
\usepackage{amsmath}
\usepackage{amsfonts}
\usepackage{amssymb,amsbsy,amsthm}
\usepackage{hyperref}
\usepackage{geometry}
\usepackage{color}
\usepackage{ytableau}
\usepackage[T1]{fontenc}
\usepackage[enableskew,vcentermath]{youngtab}
\usepackage{tikz}

	\newcommand{\red}[1]{\textcolor{red}{#1}}

\newcommand{\FF}{\mathcal F}

\newcommand{\phimapsto}{{\overset{\phi}{\mapsto}}}

\DeclareMathOperator\Des{Des}

\DeclareMathOperator\cDes{cDes}

\def\ZZ{{\mathbb Z}}

\def\M{{\mathbf M}}
\def\cS{{\mathcal S}}
 \def\su{\textsc{U}}
 \def\sd{\textsc{D}}
 \def\sl{\textsc{L}}
\def\sp{\textsc{P}}

\newcommand{\younganddescent}[2]{
    \text{
        \begin{tabular}{l}
            $\young(#1)$ \\[10px]
            #2
        \end{tabular}
}}

\newcommand{\Jdt}{\operatorname{jdt}}

\def\SYT{{\rm SYT}}
\def\Stat{{\rm Stat}}
\def\cStat{{\rm cStat}}

\def\pro{{\rm pro}}

\theoremstyle{plain}
\newtheorem{theorem}{Theorem}[section]
\newtheorem{proposition}[theorem]{Proposition}
\newtheorem{lemma}[theorem]{Lemma}
\newtheorem{corollary}[theorem]{Corollary}

\newtheorem{problem}[theorem]{Problem}
\theoremstyle{definition}
\newtheorem{definition}[theorem]{Definition}

\newtheorem{example}[theorem]{Example}

\newtheorem{remark}[theorem]{Remark}

\begin{document}
	
	\title[Cyclic descents for Motzkin paths]
	{Cyclic descents for Motzkin paths}



	\author{Bin Han}
	\address{Department of Mathematics, Royal institute of Technology (KTH), SE 100-44 Stockholm, Sweden}
	\email{binhan@kth.se, han.combin@hotmail.com}
%


	\date{\today} 
	
		\thanks{Supported by the Israel Science Foundation, grant no.\ 1970/18 as a
postdoctoral fellow at Bar Ilan University, during which most of the work was done. Currently supported by a Starting Grant (No. 2019-05195) from Vetenskapsr\aa{}det (The Swedish Research Council)} 
	

	\maketitle	
	
	\begin{abstract}
A notion of cyclic descents on standard Young tableaux (SYT) of rectangular shape was introduced by Rhoades, and extended to certain skew shapes by Adin, Elizalde and Roichman. The cyclic descent set restricts to the usual descent set when the largest value is ignored, and has the property that the number of SYT of a given shape with a given cyclic descent set D is invariant under cyclic shifts of the entries of D.  Adin, Reiner and Roichman proved that a skew shape has a cyclic descent map if and only if it is not a connected ribbon. Unfortunately, their proof is nonconstructive. Recently Huang constructed an explicit cyclic descent map for all shapes where this is possible.

In the earlier version of Adin, Elizalde and Roichman's paper, they asked to find statistics on combinatorial objects which are equidistributed with cyclic descents on SYT of given shapes. In this paper, we explicitly describe cyclic descent sets for Motzkin paths,  which are equidistributed with
cyclic descent sets of SYT for three-row shapes. Moreover, in light of Stanley's shuffling theorem, we give a bijective proof of the shuffling property of descent statistics for Motzkin paths.
	\end{abstract}

	

	\tableofcontents
	
	\section{Introduction}
	
	
	
	Let $\cS_n$ denote the symmetric group on $[n]=\{1,\ldots,n\}$.
For a permutation $\pi\in \cS_n$, the {\em descent set} of $\pi$ is defined by
\begin{equation*}
    \Des(\pi) := \{i\in [n-1]: \pi(i) > \pi(i+1)\}.
\end{equation*}
 The notion of {\em cyclic descent} was introduced by Cellini \cite{Cellini} and 
further studied by Dilks, Petersen, Stembridge~\cite{DPS}.
The cyclic descent set is defined, for $\pi \in \cS_n$, by
\begin{equation*}
    \cDes(\pi):=
    \begin{cases}
        \Des(\pi) \cup \{n\} & \pi(n) > \pi(1), \\
        \Des(\pi) & \pi(n) < \pi(1).
    \end{cases}
\end{equation*}

Moreover, let $\phi: \cS_n \mapsto \cS_n$ be the cyclic rotation function,
such that
\begin{equation*}
    (\phi \pi) (i) = \pi (i-1)
\end{equation*}
for each $\pi \in \cS_n$, where indices are taken modulo $n$.
Then $\phi$ has the property that for any $\pi \in \cS_n$,
\begin{equation}\label{prop:+1}
    \cDes(\pi)+1 = \cDes(\phi \pi),
\end{equation}
where $\cDes(\pi)+1$ denotes the set obtained from $\cDes(\pi)$
by adding $1\pmod n$ to each element. Property~\eqref{prop:+1} of the bijection $\phi$ implies that the multiset $\{\{\cDes(\pi) | \pi \in \cS_n\}\}$ is rotation-invariant modulo $n$.

%
%


\smallskip
	
	Rhoades introduced a cyclic descent notion for standard Young tableaux (SYT), but only for rectangular shapes~\cite{Rhoades}; the definition of SYT is postponed to Section~\ref{sec:SYT}.
	A generalization of this notion to 
	other skew shapes, was introduced in~\cite{AER, ARR, Huang}.

	
	The following definition of cyclic extension is due to Adin, Reiner and Roichman.
	\begin{definition}\label{def:cDes}\cite[Definition~2.1]{ARR}
		Let $\FF$ be a finite set, equipped with a {\em descent map} 
		$\Des: \FF \longrightarrow 2^{[n-1]}$. 
		A {\em cyclic extension} of $\Des$ is
		a pair $(\cDes,p)$, where 
		$\cDes: \FF \longrightarrow 2^{[n]}$ is a map 
		and $p: \FF \longrightarrow \FF$ is a bijection,
		satisfying the following axioms:  for all $F$ in  $\FF$,
		\[
		\begin{array}{rl}
		\text{(extension)}   & \cDes(F) \cap [n-1] = \Des(F),\\
		\text{(equivariance)}& \cDes(p(F))  = \cDes(F)+1, {\rm with\,\, indices\,\, taken\,\, modulo} \,\,$n$.\\
		\text{(non-Escher)}  & \varnothing \subsetneq \cDes(F) \subsetneq [n].\\
		\end{array}
		\]
	\end{definition}

		\medskip
	

\smallskip


In an earlier version of Adin, Elizalde and Roichman~\cite{AER}, they posed the following problem.
\begin{problem}
Find statistics on combinatorial objects which are equidistributed with cyclic descents
on SYT of given shapes.
\end{problem}
Motivated by the above problem, we aim to find cyclic descent sets for Motzkin paths in this paper.

A Motzkin path of length $n$ is a lattice path in the integer plane $\ZZ\times\ZZ$ from $(0,0)$ to $(n, 0)$ which consists of three kinds of steps $\su=(1, 1)$, $\sd=(1, -1)$ and $\sl=(1, 0)$ and never passes below the $x$-axis. A Dyck path is a Motzkin path without $\sl$ steps.  Recently, Adin, Elizalde and Roichman \cite[Section~5]{AER} provided an explicit combinatorial description of a cyclic descent extension  for Dyck paths. It is natural to consider a cyclic descent extension for Motzkin paths.

Let $\M_n$ denote the set of Motzkin paths with $n$ steps. A {\em return} (resp. {\em start}) of a path in $\M_n$ is a $\sd$ (resp. $\su$) step that ends (resp. starts) on the $x$-axis. 

Consider the order $\su>\sd>\sl$ on Motzkin path steps. This order induces a descent map $\Des: \M_n\mapsto 2^{[n-1]}$ defined as follows.

\begin{definition}\label{def:motzdes}
The step $i\in[n]$ is a  {\em descent} of a path $M\in\M_n$ if one of the following conditions holds:
\begin{enumerate}
\item
 The $i$th step of $M$ is a $\su$ followed by a $\sd$.
\item
 The $i$th step of $M$ is a $\su$ followed by a $\sl$.
 \item
  The $i$th step of $M$ is a $\sd$ followed by a $\sl$.
\end{enumerate}
Let $\Des\,(M)$ denote the set of descents of $M$.
\end{definition}

\begin{remark}
\begin{enumerate}
\item
For Motzkin paths without $\sl$ steps, descent sets in Definition~\ref{def2:motzdes} coincide with  those of  Dyck paths, see~\cite[Definition~5.1]{AER}.

\item
Let 
$T_3(n)$ denote the set of SYTs with $n$ entries and at most $3$ rows. Regev~\cite{Reg81} proved that $|T_3(n)|$ is equal to the $n$-th Motzkin number.
Eu~\cite{Eu10} provided a bijection between $T_3(n)$ and $\M_n$, but it does not preserve the descent set in general.  We give a bijection between  $T_3(n)$ and $\M_n$ preserving the descent set. 

\item
We construct a cyclic extension of this descent map. In fact, we achieve this for all of the orderings $\su>\sd>\sl$, $\su>\sl>\sd$ and $\sl>\su>\sd$.

\end{enumerate}
\end{remark}

For a finite set (alphabet) $A$ of size $a$, let $\cS_A$ be the set of all bijections $u : [a] \to A$, viewed as words $u =(u_1,\ldots, u_a)$. 
If $A=[a]$ then $\cS_A$ is the symmetric group $\cS_a$, whose elements are permutations of $[a]$.

Richard Stanley's theory of P-partitions~\cite{Sta72} implies that the descent set statistic has a remarkable property related to shuffles: for any two permutations $\pi$ and $\sigma$ on disjoint alphabet $s$, the descent set over all shuffles of $\pi$ and $\sigma$ depends only on $\Des(\pi)$, $\Des(\sigma)$,  and the lengths of $\pi$ and $\sigma$~\cite[Exercise~3.161]{EC1}. The following classical result is due to Stanley.

\begin{lemma}(\cite[Exercise~3.161]{EC1} and \cite[Section~2.4]{GZ18})\label{lem:shuffle}
Let A and B be two disjoint sets of integers. For each $u\in \cS_A$ and $v\in \cS_B$, the distribution of the descent set over all shuffles of $u$ and $v$ depends only on $\Des(u)$, $\Des(v)$ and the lengths of $u$ and $v$.
\end{lemma}

Every Motzkin path is a shuffle of a Dyck path and a horizontal path. By suitable labelings, monotone on each type of step separately, we can view them as permutations. Applying Lemma~\ref{lem:shuffle}, we obtain the following proposition immediately.

\begin{proposition}\label{prop:1.6}
The descent distributions on all Motzkin paths are the same for the orderings $\su>\sd>\sl$, $\su>\sl>\sd$ and $\sl>\su>\sd$.
\end{proposition}
We give a bijective proof for the above result in Section~\ref{sec:Mcyc}.

For motzkin paths with the order $\su>\sd>\sl$, we define the  cyclic descent extension as follows. 


\begin{definition}\label{def:Mcdes}
For $M\in\M_n$, let $n\in{\cDes}(M)$ if and only if either of the following conditions holds:
\begin{enumerate}
\item The first step of $M$ is $\sl$ and the final step of $M$ is $\sd$.
\item The first step of $M$ is $\su$, the final step of $M$ is $\sd$ and $M$ has no returns before the rightmost $\sd$.
\end{enumerate}

\end{definition}

\begin{remark}\label{remark:defmotz}
\begin{enumerate}
\item
When $M\in\M_n$ has no $\sl$ steps, Definition~\ref{def:Mcdes} coincides with that of Adin, Elizalde and Roichman for Dyck paths, see \cite[Definition~5.7]{AER}.
\item
It is clear from Definition~\ref{def:motzdes} and Definition~\ref{def:Mcdes} that having $\{n-1,n\}\subseteq\cDes(M)$ forces $1, n-2\notin \cDes(M)$ or $\{n, 1\}\subseteq\cDes(M)$ forces $n-1, 2\notin \cDes(M)$. That is consistent with the fact that a Motzkin path can not have three descents in consecutive positions.
\end{enumerate}
\end{remark}


\begin{example}
Figure~\ref{cdesmot2} shows two instances of $M\in\M_6$ for which $6\in {\cDes}(M)$.
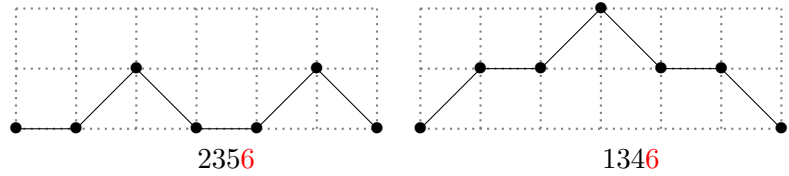
\begin{figure}[htpb]
\begin{center}
\begin{tikzpicture}[scale=0.8]
		\draw[step=1cm, gray, thick,dotted] (0, 0) grid (6,2);
 		
 		\draw[black] (0, 0)--(1, 0)--(2, 1)--(3, 0)--(4, 0)--(5, 1)--(6, 0);
 		\draw[black] (0,0) node {$\bullet$};
 		\draw[black] (1,0) node {$\bullet$};
 		\draw[black] (2,1) node {$\bullet$};
 		\draw[black] (3,0) node {$\bullet$};
 		\draw[black] (4,0) node {$\bullet$};
 		\draw[black] (5,1) node {$\bullet$};
 		\draw[black] (6,0) node {$\bullet$};
                 \draw[black] (3.5,-0.5) node {$235\red{6}$};
 		\end{tikzpicture}
\begin{tikzpicture}[scale=0.8]
 		\draw[step=1cm, gray, thick,dotted] (0, 0) grid (6,2);
 		
 		\draw[black] (0, 0)--(1, 1)--(2, 1)--(3, 2)--(4, 1)--(5, 1)--(6, 0);
 		\draw[black] (0,0) node {$\bullet$};
 		\draw[black] (1,1) node {$\bullet$};
 		\draw[black] (2,1) node {$\bullet$};
 		\draw[black] (3,2) node {$\bullet$};
 		\draw[black] (4,1) node {$\bullet$};
 		\draw[black] (5,1) node {$\bullet$};
 		\draw[black] (6,0) node {$\bullet$};
                 \draw[black] (3.5,-0.5) node {$134\red{6}$};
 		\end{tikzpicture}			
\caption{Cyclic descent sets of two Motzkin paths in $\M_6$.}\label{cdesmot2}
\end{center}
\end{figure}
\end{example}





Let $\sl^n$ denote the path in $\M_n$ whose steps are all $\sl$ steps. Because of the non-Escher property of a cyclic extension, i.e., $\cDes(M)\neq \emptyset$ for $M\in\M_n$, we only consider a cyclic descent extension on  $\M^*_n:=\M_n\setminus{\{\sl^n\}}$.  For any fixed $0\leq k<n$, define
$$\M_{n,k}:=\{M\in\M_n: |M|_\sl=k\},$$
where $|M|_\sl$ is the number of $\sl$ steps in a Motzkin path $M$.
Our main result for Motzkin paths is stated next; the precise definition of $\rho$ is postponed to Section~\ref{sec:result1}. 


\begin{theorem}\label{cdesmotz}
For any $n\ge 3$ and $0\le k<n$, the pair $({\cDes}, \rho)$, with $\cDes$ and $\rho$ determined by Definition~\ref{def:Mcdes} and Definition~\ref{def:motzrho}, respectively, is a cyclic extension of ${\Des}$ on $\M_{n,k}$.
\end{theorem}
Note that for $n\le 1$, $M_n^*=\varnothing$; and for $n=2$, $M_2^*=\{UD\}$ has a single element and therefore $\Des$ has no (non-Escher) cyclic extension.
By Theorem~\ref{cdesmotz}, we obtain the following corollary immediately.
\begin{corollary}\label{coro:cde} 
For any $n\ge 3$, the pair $({\cDes}, \rho)$ is a cyclic extension of $\Des$ on $\M^*_{n}$.
\end{corollary}
The paper is organized as follows. In Section~\ref{sec:result1}, we prove Theorem~\ref{cdesmotz}.
In Section~\ref{sec:Mcyc}, we give another equivalent cyclic extension for Motzkin paths.
 In Section~\ref{sec:SYT}, we describe the cyclic extension of Motzkin paths in terms of standard Young tableaux.

 
 \section{Proof of Theorem~\ref{cdesmotz}}\label{sec:result1}

Observation: $\cDes(M)=\varnothing$ if and only if $M=L^n$.  $\cDes(M)=[n]$ if and only if $M=\su\sd$ (and $n=2$). 
We construct a bijection $\rho:\M^*_n\to\M^*_n$ such that
${\cDes}(\rho M)=1+{\cDes}(M)$ for all $M\in\M^*_n\, (n\ge 3)$, with addition modulo $n$. Our construction of $(\cDes, \rho)$ comes from the idea that for all $1\leq i\leq n$, we have that $i\in\cDes(M)$ if and only if $i+1\in\cDes(\rho M)  \pmod n$.

\begin{definition}\label{def:motzrho}
Given $M\in\M^*_n$, define $\rho M$ by considering two cases:
\begin{enumerate}
	\item The final step of $M$ is $\sl$. Let $\rho M$ be the path obtained by moving the last $\sl$ of $M$ and inserting it as the first step of the path. In other words, write $M=\sp\sl$ for any Motzkin path $\sp\neq\sl^{n-1}$, and let $\rho M=\sl\sp$, 
see Figure~\ref{case2} for an example.
	 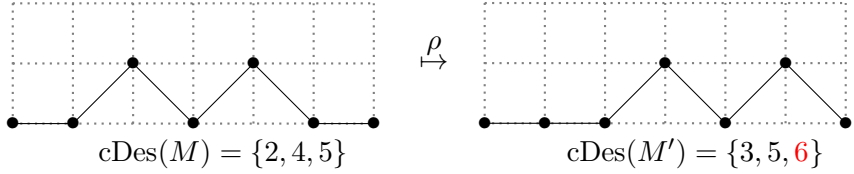
\begin{figure}
	 \begin{center}
	 		\begin{tikzpicture}[scale=0.8]
 		\draw[step=1cm, gray, thick,dotted] (0, 0) grid (6,2);
 		
 		\draw[black] (0, 0)--(1, 0)--(2, 1)--(3, 0)--(4, 1)--(5, 0)--(6, 0);
 		\draw[black] (0,0) node {$\bullet$};
 		\draw[black] (1,0) node {$\bullet$};
 		\draw[black] (2,1) node {$\bullet$};
 		\draw[black] (3,0) node {$\bullet$};
 		\draw[black] (4,1) node {$\bullet$};
 		\draw[black] (5,0) node {$\bullet$};
 		\draw[black] (6,0) node {$\bullet$};
                 \draw[black] (3.5,-0.5) node {$\cDes(M)=\{2,4,5\}$};
                \draw[black] (7.0, 1) node {$\mapsto$};
              \draw[black](7.0, 1.3) node {$\rho$};
 		\end{tikzpicture}
		\begin{tikzpicture}[scale=0.8]
 		\draw[step=1cm, gray, thick,dotted] (0, 0) grid (6,2);
 		\draw[black] (0, 0)--(1, 0)--(2, 0)--(3, 1)--(4, 0)--(5, 1)--(6, 0);
 		\draw[black] (0,0) node {$\bullet$};
 		\draw[black] (1,0) node {$\bullet$};
 		\draw[black] (2,0) node {$\bullet$};
 		\draw[black] (3,1) node {$\bullet$};
 		\draw[black] (4,0) node {$\bullet$};
 		\draw[black] (5,1) node {$\bullet$};
 		\draw[black] (6,0) node {$\bullet$};
                 \draw[black] (3.5,-0.5) node {$\cDes(M')=\{3,5,\red{6}\}$};
 		\end{tikzpicture}
		\caption{Case\,1:\,\,\,\,\,\,$\rho(M)=M'$}\label{case2}		 
	 \end{center}
	 \end{figure}

	\item  The final step of $M$ is $\sd$.
	 Let $\rho M$ be the path obtained by two steps. 
	 \begin{enumerate}
	\item First, move the last $\sd$ of $M$ and insert it before the final start step of the path.     
	\item  Second, move the last start step $\su$ of $M$ and insert it as the first step of the path. 
	  \end{enumerate}

	Write $M=\sp\su\sp'\sd$, where $\su$ is the final start step of $M$, 
	$\sp$ and $\sp'$ are any (possibly empty) Motzkin paths. 
	Let $\rho M=\su\sp\sd\sp'$, see Figure~\ref{case7} for an example.
	     	 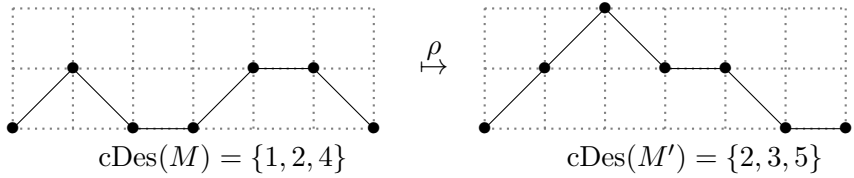
\begin{figure}
	 \begin{center}
	 		\begin{tikzpicture}[scale=0.8]
 		\draw[step=1cm, gray, thick,dotted] (0, 0) grid (6,2);
 		
 		\draw[black] (0, 0)--(1, 1)--(2, 0)--(3, 0)--(4, 1)--(5, 1)--(6, 0);
 		\draw[black] (0,0) node {$\bullet$};
 		\draw[black] (1,1) node {$\bullet$};
 		\draw[black] (2,0) node {$\bullet$};
 		\draw[black] (3,0) node {$\bullet$};
 		\draw[black] (4,1) node {$\bullet$};
 		\draw[black] (5,1) node {$\bullet$};
 		\draw[black] (6,0) node {$\bullet$};
                 \draw[black] (3.5,-0.5) node {$\cDes(M)=\{1,2,4\}$};
                \draw[black] (7.0, 1) node {$\mapsto$};
              \draw[black](7.0, 1.3) node {$\rho$};
 		\end{tikzpicture}
		\begin{tikzpicture}[scale=0.8]
 		\draw[step=1cm, gray, thick,dotted] (0, 0) grid (6,2);
 		\draw[black] (0, 0)--(1, 1)--(2, 2)--(3, 1)--(4, 1)--(5, 0)--(6, 0);
 		\draw[black] (0,0) node {$\bullet$};
 		\draw[black] (1,1) node {$\bullet$};
 		\draw[black] (2,2) node {$\bullet$};
 		\draw[black] (3,1) node {$\bullet$};
 		\draw[black] (4,1) node {$\bullet$};
 		\draw[black] (5,0) node {$\bullet$};
 		\draw[black] (6,0) node {$\bullet$};
                 \draw[black] (3.5,-0.5) node {$\cDes(M')=\{2,3,5\}$};
 		\end{tikzpicture}
		\caption{Case\,2:\,\,\,\,\,\,$\rho(M)=M'$}\label{case7}		 
	 \end{center}
	 \end{figure}

\end{enumerate}
\end{definition}	


%


Observing all two cases of Definition~\ref{def:motzrho}, we obtain the following property of $\rho$.	
	
\begin{proposition}\label{Mot:level=}
The number of horizontal steps is invariant under the shift operation $\rho$ on Motzkin paths.
\end{proposition}

\begin{proof}[Proof of Theorem~\ref{cdesmotz}]	
By Proposition~\ref{Mot:level=} $\rho$ maps $M_{n,k}$ into itself. To prove Theorem~\ref{cdesmotz} we need to  show that $(\cDes, \rho)$ satisfies the extension, non-Escher and equivariance.
Extension follows from Definition~\ref{def:motzdes} and Definition~\ref{def:Mcdes}. 
For $M\in\M^*_n$, $M$ has at least one $\su$ step (which cannot be the last step), and thus by Definition~\ref{def:motzdes}, we have $1\leq|\cDes(M)|$. By Remark~\ref{remark:defmotz}, we have, for $n\ge 3$, $|\cDes(M)| \leq n-1$. This implies the non-Escher property.
To  prove the equivariance property, we need to show that $\cDes(\rho M)=1+\cDes(M)$, with addition modulo $n$, in all two cases of Definition~\ref{def:motzrho}. For example, in case $1$, we have that $n \notin\cDes(M)$ and $1\notin\cDes(\rho M)$(by Definition~\ref{def:motzdes} and Definition~\ref{def:Mcdes}), whereas $n-1 \in\cDes(M)$ if and only if the final step of $\sp$ is $\sd$ which implies $n\in\cDes(\rho M)$; additionally, for all $1\leq i\leq n-2$, we have that $i\in\cDes(M)$ if and only if $i+1\in\cDes(\rho M)$ because all such descents must be inside $\sp$, which is shifted one position to the right by $\rho$. Similarly for the other case.

To prove that $\rho$ is a bijection, we describe its inverse in each of the two cases above. For a given $M\in\M^*_n$, define $\rho^{-1}M$ as follows:

\begin{enumerate}
	\item The first step of $M$ is $\sl$. Let $\rho^{-1} M$ be the path obtained by moving the first $\sl$ of $M$ and inserting it as the last step of the path. In other words, write $M=\sl\sp$ for any Motzkin path $\sp\neq\sl^{n-1}$, and let $\rho^{-1} M=\sp\sl$. 

	\item  The first step of $M$ is $\su$.
	 Let $\rho^{-1} M$ be the path obtained by two steps. 
	 \begin{enumerate}
	\item First, move the first $\su$ of $M$ and insert it before the first return step of the path.     
	\item  Second, move the first return step $\sd$ of $M$ and insert it as the final step of the path. 
	  \end{enumerate}

	In other words, write $M=\su\sp\sd\sp'$, where $\sd$ is the first return step of $M$, 
	$\sp$ and $\sp'$ are any (possibly empty) Motzkin paths.  Then we define $\rho^{-1} M=\sp\su\sp'\sd$.

\end{enumerate}

\end{proof}


\section{A cyclic extension for a different order}\label{sec:Mcyc}

In this section, we construct a cyclic descent extension of the descent map on Motzkin paths, for a different order between the $\su, \sl$ and $\sd$ step types. We first introduce the following two bijections, which preserve the descent set under different orderings of the three step types $\su$, $\sl$, $\sd$.

\begin{lemma}\label{bijection:Mot}
\begin{enumerate}
\item
For $M\in\M_n$, there is a descent set preserving bijection~$\varphi$ on Motzkin paths which transforms $\su>\sl>\sd$ to $\su>\sd>\sl$ and preserves the numbers of $\su$, $\sl$ and $\sd$ steps.
\item
For $M\in\M_n$, there is a descent set preserving bijection~$\varphi'$ on Motzkin paths which transforms $\sl>\su>\sd$ to $\su>\sl>\sd$ and preserves the numbers of $\su$, $\sl$ and $\sd$ steps.
\end{enumerate}
\end{lemma}
\begin{proof}
The map~$\varphi: M\to M$ can be described as follows.
Preserve the locations of all $\su$'s and all other descents ($\sl\sd$, which transform to $\sd\sl$).
In all intermediate intervals (between $\su$'s and descents) reverse the (descent-free) sequence of letters:  $\sd^a\sl^b\to \sl^b\sd^a$. It is easy to see  that $\varphi$ is a bijection by considering the inverse map. Clearly $\varphi$ preserves the numbers of $\su$, $\sl$ and $\sd$ steps.

Similarly, 
The map~$\varphi': M\to M$ can be described as follows.
Preserve the locations of all $\sd$'s and all other descents ($\sl\su$, which transform to $\su\sl$).
In all intermediate intervals (between $\sd$'s and descents) reverse the (descent-free) sequence of letters:  $\su^a\sl^b\to \sl^b\su^a$. It is easy to see $\varphi'$ is a bijection by considering the inverse map. Clearly $\varphi'$ preserves the numbers of $\su$, $\sl$ and $\sd$ steps.
\end{proof}

\begin{example}
For $M\in\M_7$, 
we have
\begin{align*}
&\su\su\sd\sd\sl\sl\sl \xrightarrow[]{\varphi}
\su\su\sl\sl\sl\sd\sd \xrightarrow[]{\varphi}
\su\su\sl\sl\sd\sl\sd\\ 
&\xrightarrow[]{\varphi} 
\su\su\sl\sd\sl\sd\sl 
\xrightarrow[]{\varphi}
 \su\su\sd\sl\sd\sl\sl 
(\xrightarrow[]{\varphi}\su\su\sd\sd\sl\sl\sl),
\end{align*}
and
\begin{align*}
&\su\su\sl\sl\sl\sd\sd \xrightarrow[]{\varphi'}
\sl\sl\sl\su\su\sd\sd \xrightarrow[]{\varphi'}
\sl\sl\su\sl\su\sd\sd\\ 
&\xrightarrow[]{\varphi'} 
\sl\su\sl\su\sl\sd\sd 
\xrightarrow[]{\varphi'}
 \su\sl\su\sl\sl\sd\sd 
(\xrightarrow[]{\varphi'}\su\su\sl\sl\sl\sd\sd).
\end{align*}
\end{example}
Lemma~\ref{bijection:Mot} yields a bijective proof of Proposition~\ref{prop:1.6}.

%
%
%
By Corollary~\ref{coro:cde} and Proposition~\ref{prop:1.6}, we have the following result.
\begin{proposition}
The descent map on $\M^*_n$ has a cyclic extension for each of the orders $\su>\sd>\sl$, $\su>\sl>\sd$ and $\sl>\su>\sd$.
\end{proposition}

We shall find an explicit construction of a cyclic descent extension for the order $\su>\sl>\sd$. The construction for $\sl>\su>\sd$ is left to the interested reader.



%


\begin{definition}\label{def2:motzdes}

Assume that $\su>\sl>\sd$. An index $i$ is a descent of a path $M\in\M_n$ if one of the following conditions holds:
\begin{enumerate}
\item
 The $i$th step of $M$ is a $\su$ followed by a $\sd$.
\item
 The $i$th step of $M$ is a $\su$ followed by an $\sl$.
 \item
  The $i$th step of $M$ is an $\sl$ followed by a $\sd$.
\end{enumerate}
Let $\widehat{\Des}(M)$ denote the set of descents of $M$.
\end{definition}

We now define the following cyclic descents.


 \begin{definition}\label{def2:Mcdes}
For $M\in\M_n$, let $n\in\widehat{\cDes}(M)$ if and only if either of the following conditions holds:
\begin{enumerate}
\item  The first step of $M$ is $\sl$ and the step preceding the rightmost return is not an $\sl$.
\item The first step of $M$ is $\su$, $M$ has only one return, and the step preceding the return is not an $\sl$.
\end{enumerate}

\end{definition}

\begin{example}
By Definition~\ref{def2:motzdes} and Definition~\ref{def2:Mcdes}, we give two examples in Figure~\ref{2edcdesmot2}.
\begin{figure}[htpb]
\begin{center}
\begin{tikzpicture}[scale=0.8]
		\draw[step=1cm, gray, thick,dotted] (0, 0) grid (6,2);
 		
 		\draw[black] (0, 0)--(1, 0)--(2, 1)--(3, 0)--(4, 1)--(5, 0)--(6, 0);
 		\draw[black] (0,0) node {$\bullet$};
 		\draw[black] (1,0) node {$\bullet$};
 		\draw[black] (2,1) node {$\bullet$};
 		\draw[black] (3,0) node {$\bullet$};
 		\draw[black] (4,1) node {$\bullet$};
 		\draw[black] (5,0) node {$\bullet$};
 		\draw[black] (6,0) node {$\bullet$};
                 \draw[black] (3.5,-0.5) node {$24\red{6}$};
 		\end{tikzpicture}
\begin{tikzpicture}[scale=0.8]
 		\draw[step=1cm, gray, thick,dotted] (0, 0) grid (6,2);
 		
 		\draw[black] (0, 0)--(1, 1)--(2, 1)--(3, 2)--(4, 2)--(5, 1)--(6, 0);
 		\draw[black] (0,0) node {$\bullet$};
 		\draw[black] (1,1) node {$\bullet$};
 		\draw[black] (2,1) node {$\bullet$};
 		\draw[black] (3,2) node {$\bullet$};
 		\draw[black] (4,2) node {$\bullet$};
 		\draw[black] (5,1) node {$\bullet$};
 		\draw[black] (6,0) node {$\bullet$};
                 \draw[black] (3.5,-0.5) node {$134\red{6}$};
 		\end{tikzpicture}			
\caption{Cyclic descent sets of two Motzkin paths in $\M_6$}\label{2edcdesmot2}
\end{center}
\end{figure}
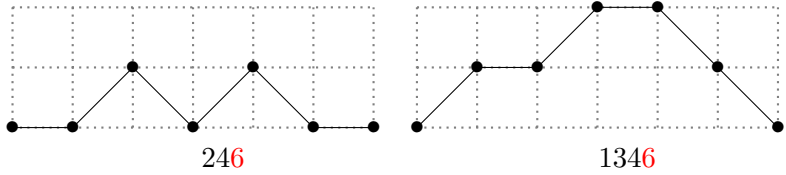
\end{example}

We construct a  bijection $\hat{\rho}:\M^*_n\to\M^*_n$ such that
$\cDes(\hat\rho M)=1+\cDes(M)$ for all $M\in\M^*_n\, (n\ge 3)$, where addition is modulo $n$.

\begin{definition}\label{def:rho1}

Given $M\in\M^*_n$, define $\hat\rho M$ by considering three cases:

\begin{enumerate}
\item If the step before the last return of $M$ is $\su$. Write $M=\sp\sl^i\su\sd\sl^j\,(i, j\geq 0)$, where $\sp$ may be empty or the final step of $\sp\in\M_{n-2-i-j}$ is $\sd$. There are two cases.
\begin{enumerate}
\item When $j=0$, let $\hat\rho M$ be the path obtained by two steps.
	 \begin{enumerate}
	 \item First, move the last $\su$ of $M$ and insert it as the first step of the path.     
	\item Second,  move the last $\sd$ of $M$ and insert it after the second $\sd$ step of the path from right to left. 	 
	\end{enumerate}
In other words, let $\hat\rho M=\su \sp\sd\sl^i$. See Figure~\ref{2edcase1} for an example.  

	 \begin{figure}
	 \begin{center}
	 		\begin{tikzpicture}[scale=0.8]
 		\draw[step=1cm, gray, thick,dotted] (0, 0) grid (6,2);
 		
 		\draw[black] (0, 0)--(1, 0)--(2, 1)--(3, 0)--(4, 0)--(5, 1)--(6, 0);
 		\draw[black] (0,0) node {$\bullet$};
 		\draw[black] (1,0) node {$\bullet$};
 		\draw[black] (2,1) node {$\bullet$};
 		\draw[black] (3,0) node {$\bullet$};
 		\draw[black] (4,0) node {$\bullet$};
 		\draw[black] (5,1) node {$\bullet$};
 		\draw[black] (6,0) node {$\bullet$};
                 \draw[black] (3.5,-0.5) node {$\cDes(M)=\{2,5,\red{6}\}$};
                \draw[black] (7.0, 1) node {$\mapsto$};
              \draw[black](7.0, 1.3) node {$\hat\rho$};
 		\end{tikzpicture}
		\begin{tikzpicture}[scale=0.8]
 		\draw[step=1cm, gray, thick,dotted] (0, 0) grid (6,2);
 		\draw[black] (0, 0)--(1, 1)--(2, 1)--(3, 2)--(4, 1)--(5, 0)--(6, 0);
 		\draw[black] (0,0) node {$\bullet$};
 		\draw[black] (1,1) node {$\bullet$};
 		\draw[black] (2,1) node {$\bullet$};
 		\draw[black] (3,2) node {$\bullet$};
 		\draw[black] (4,1) node {$\bullet$};
 		\draw[black] (5,0) node {$\bullet$};
 		\draw[black] (6,0) node {$\bullet$};
                 \draw[black] (3.5,-0.5) node {$\cDes(M')=\{1,3,\red{6}\}$};
 		\end{tikzpicture}
		\caption{Case\,1(a):\,\,\,\,\,\,$\hat\rho(M)=M'$}\label{2edcase1}		 
	 \end{center}
	 \end{figure}
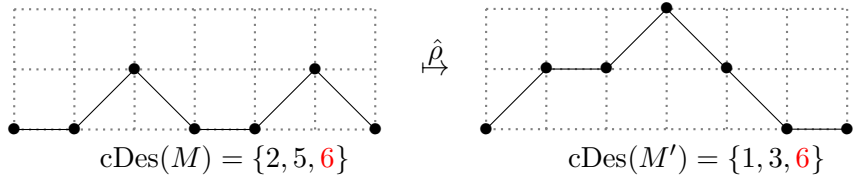

\item When $j\ge1$, let $\hat\rho M=\su \sp\sl^{i+1}\sd\sl^{j-1}$.  See Figure~\ref{2edcase3c} for an example.
%
%
%
%
	 \begin{figure}
	 \begin{center}
	 		\begin{tikzpicture}[scale=0.8]
 		\draw[step=1cm, gray, thick,dotted] (0, 0) grid (6,2);
 		
 		\draw[black] (0, 0)--(1, 0)--(2, 1)--(3, 0)--(4, 1)--(5, 0)--(6,0);
 		\draw[black] (0,0) node {$\bullet$};
 		\draw[black] (1,0) node {$\bullet$};
 		\draw[black] (2,1) node {$\bullet$};
 		\draw[black] (3,0) node {$\bullet$};
 		\draw[black] (4,1) node {$\bullet$};
 		\draw[black] (5,0) node {$\bullet$};
 		\draw[black] (6,0) node {$\bullet$};
                 \draw[black] (3.5,-0.5) node {$\cDes(M)=\{2, 4, \red{6}\}$};
                \draw[black] (7.0, 1) node {$\mapsto$};
              \draw[black](7.0, 1.3) node {$\hat\rho$};
 		\end{tikzpicture}
		\begin{tikzpicture}[scale=0.8]
 		\draw[step=1cm, gray, thick,dotted] (0, 0) grid (6,2);
 		\draw[black] (0, 0)--(1, 1)--(2, 1)--(3, 2)--(4, 1)--(5, 1)--(6, 0);
 		\draw[black] (0,0) node {$\bullet$};
 		\draw[black] (1,1) node {$\bullet$};
 		\draw[black] (2,1) node {$\bullet$};
 		\draw[black] (3,2) node {$\bullet$};
 		\draw[black] (4,1) node {$\bullet$};
 		\draw[black] (5,1) node {$\bullet$};
 		\draw[black] (6,0) node {$\bullet$};
	        \draw[black] (3.5,-0.5) node {$\cDes(M')=\{1, 3, 5\}$};
 		\end{tikzpicture}
		\caption{Case\,1(b):\,\,\,\,\,\,$\hat\rho(M)=M'$}\label{2edcase3c}		 
	 \end{center}
	 \end{figure}
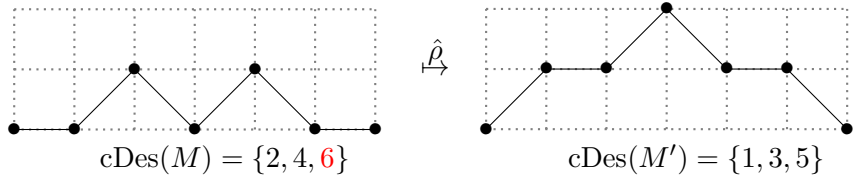

\end{enumerate}

\item
If the step before the last return of $M$ is $\sl$. 
 Write $M=\sp\sl^i\sd\sl^j\,(i\geq 1, j\geq 0)$, where the final step of $\sp\in\M_{n-1-i-j}$ is not $\sl$. There are two cases. 
  \begin{enumerate}
	 \item When $j=0$, let $\hat\rho M$ be the path obtained by two steps. 
	 \begin{enumerate}
	\item First, move the last $\sl$ of $M$ and insert it as the first step of the path.     
	\item  Second, move the last $\sd$ of $M$ and insert it after the second non $\sl$ step of the path from right to left. 
	  \end{enumerate}

 In other words, let $\hat\rho M=\sl \sp\sd\sl^{i-1}$. See Figure~\ref{2edcase2a} for an example. 

	 \begin{figure}
	 \begin{center}
	 		\begin{tikzpicture}[scale=0.8]
 		\draw[step=1cm, gray, thick,dotted] (0, 0) grid (6,2);
 		
 		\draw[black] (0, 0)--(1, 0)--(2, 1)--(3, 0)--(4, 1)--(5, 1)--(6, 0);
 		\draw[black] (0,0) node {$\bullet$};
 		\draw[black] (1,0) node {$\bullet$};
 		\draw[black] (2,1) node {$\bullet$};
 		\draw[black] (3,0) node {$\bullet$};
 		\draw[black] (4,1) node {$\bullet$};
 		\draw[black] (5,1) node {$\bullet$};
 		\draw[black] (6,0) node {$\bullet$};
                 \draw[black] (3.5,-0.5) node {$\cDes(M)=\{2,4,5\}$};
                \draw[black] (7.0, 1) node {$\mapsto$};
              \draw[black](7.0, 1.3) node {$\hat\rho$};
 		\end{tikzpicture}
		\begin{tikzpicture}[scale=0.8]
 		\draw[step=1cm, gray, thick,dotted] (0, 0) grid (6,2);
 		\draw[black] (0, 0)--(1, 0)--(2, 0)--(3, 1)--(4, 0)--(5, 1)--(6, 0);
 		\draw[black] (0,0) node {$\bullet$};
 		\draw[black] (1,0) node {$\bullet$};
 		\draw[black] (2,0) node {$\bullet$};
 		\draw[black] (3,1) node {$\bullet$};
 		\draw[black] (4,0) node {$\bullet$};
 		\draw[black] (5,1) node {$\bullet$};
 		\draw[black] (6,0) node {$\bullet$};
                 \draw[black] (3.5,-0.5) node {$\cDes(M')=\{3,5,\red{6}\}$};
 		\end{tikzpicture}
		\caption{Case\,2(a):\,\,\,\,\,\,$\hat\rho(M)=M'$}\label{2edcase2a}		 
	 \end{center}
	 \end{figure}
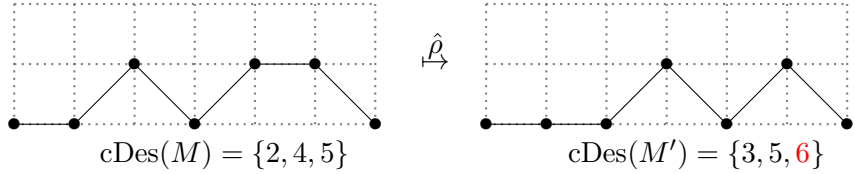


\item
when $j\ge1$, Let $\hat\rho M$ be the path obtained by moving the last $\sl$ of $M$ and inserting it as the first step of the path. In other words, let $\hat\rho M=\sl\sp\sl^{i}\sd\sl^{j-1}$. See Figure~\ref{2edcase4a} for an example.

	 \begin{figure}
	 \begin{center}
	 		\begin{tikzpicture}[scale=0.8]
 		\draw[step=1cm, gray, thick,dotted] (0, 0) grid (6,2);
 		
 		\draw[black] (0, 0)--(1, 1)--(2, 2)--(3, 1)--(4, 1)--(5, 0)--(6, 0);
 		\draw[black] (0,0) node {$\bullet$};
 		\draw[black] (1,1) node {$\bullet$};
 		\draw[black] (2,2) node {$\bullet$};
 		\draw[black] (3,1) node {$\bullet$};
 		\draw[black] (4,1) node {$\bullet$};
 		\draw[black] (5,0) node {$\bullet$};
 		\draw[black] (6,0) node {$\bullet$};
                 \draw[black] (3.5,-0.5) node {$\cDes(M)=\{2,4\}$};
                \draw[black] (7.0, 1) node {$\mapsto$};
              \draw[black](7.0, 1.3) node {$\hat\rho$};
 		\end{tikzpicture}
		\begin{tikzpicture}[scale=0.8]
 		\draw[step=1cm, gray, thick,dotted] (0, 0) grid (6,2);
 		\draw[black] (0, 0)--(1, 0)--(2, 1)--(3, 2)--(4, 1)--(5, 1)--(6, 0);
 		\draw[black] (0,0) node {$\bullet$};
 		\draw[black] (1,0) node {$\bullet$};
 		\draw[black] (2,1) node {$\bullet$};
 		\draw[black] (3,2) node {$\bullet$};
 		\draw[black] (4,1) node {$\bullet$};
 		\draw[black] (5,1) node {$\bullet$};
 		\draw[black] (6,0) node {$\bullet$};
                 \draw[black] (3.5,-0.5) node {$\cDes(M')=\{3,5\}$};
 		\end{tikzpicture}
		\caption{Case\,2(b):\,\,\,\,\,\,$\hat\rho(M)=M'$}\label{2edcase4a}		 
	 \end{center}
	 \end{figure}
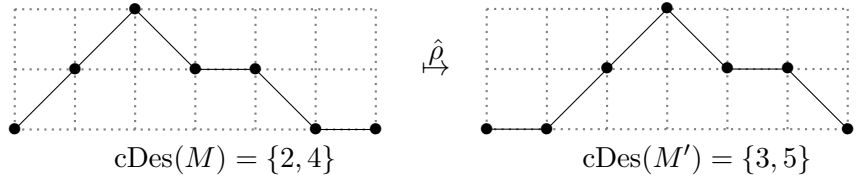	

\end{enumerate}
\item 
	The step before the last return of $M$ is $\sd$.
	 Write $M=\sp\sl^i\su\sl^k\sp'\sd\sl^j\,(i, j,k\geq 0)$, where $\su$ is the final start $\su$ step, the final step of $\sp\in\M_{n-2-|\sp'|-i-j-k}$ (possibly empty) is $\sd$, and the first step and the  final step of $\sp'$ are $\su$ and $\sd$. There are two cases. 
 \begin{enumerate}
	 \item
	When $k=0$, let $\hat\rho M$ be the path obtained by two steps. 
	 \begin{enumerate}
	 \item  First, move the final $\sd$ of $M$ and insert it after the second return of the path from right to left. 
	 \item Second, move the final start $\su$ step of $M$ and insert it as the first step of the path.  
\end{enumerate}
In other words, let $\hat\rho M=\su \sp\sd\sl^i\sp'\sl^{j}$ . See Figure~\ref{2edcase4c} for an example.  	 
	 \begin{figure}
	 \begin{center}
	 		\begin{tikzpicture}[scale=0.8]
 		\draw[step=1cm, gray, thick,dotted] (0, 0) grid (7,2);
 		
 		\draw[black] (0, 0)--(1, 1)--(2, 0)--(3, 0)--(4, 1)--(5, 2)--(6, 1)--(7,0);
 		\draw[black] (0,0) node {$\bullet$};
 		\draw[black] (1,1) node {$\bullet$};
 		\draw[black] (2,0) node {$\bullet$};
 		\draw[black] (3,0) node {$\bullet$};
 		\draw[black] (4,1) node {$\bullet$};
 		\draw[black] (5,2) node {$\bullet$};
 		\draw[black] (6,1) node {$\bullet$};
		\draw[black] (7,0) node {$\bullet$};		
                 \draw[black] (3.5,-0.5) node {$\cDes(M)=\{1, 5\}$};
                \draw[black] (8.0, 1) node {$\mapsto$};
              \draw[black](8.0, 1.3) node {$\hat\rho$};
 		\end{tikzpicture}
		\begin{tikzpicture}[scale=0.8]
 		\draw[step=1cm, gray, thick,dotted] (0, 0) grid (7,2);
 		\draw[black] (0, 0)--(1, 1)--(2, 2)--(3, 1)--(4, 0)--(5, 0)--(6, 1)--(7,0);
 		\draw[black] (0,0) node {$\bullet$};
 		\draw[black] (1,1) node {$\bullet$};
 		\draw[black] (2,2) node {$\bullet$};
 		\draw[black] (3,1) node {$\bullet$};
 		\draw[black] (4,0) node {$\bullet$};
 		\draw[black] (5,0) node {$\bullet$};
 		\draw[black] (6,1) node {$\bullet$};
		\draw[black] (7,0) node {$\bullet$};
	        \draw[black] (3.5,-0.5) node {$\cDes(M')=\{2,6\}$};
 		\end{tikzpicture}
		\caption{Case\,3(a):\,\,\,\,\,\,$\hat\rho(M)=M'$}\label{2edcase4c}		 
	 \end{center}
	 \end{figure}
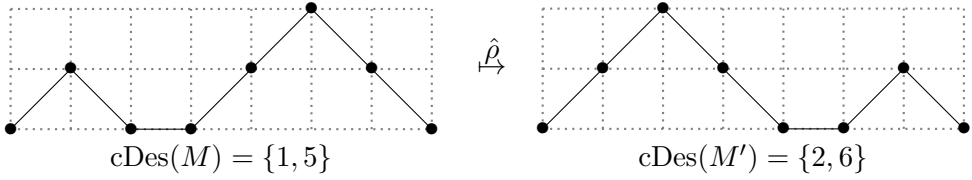

\item
When $k\ge1$, let $\hat\rho M=\su \sp\sl^{i+1}\sd\sl^{k-1}\sp'\sl^{j}$. See Figure~\ref{2edcase3a} for an example.

	
	 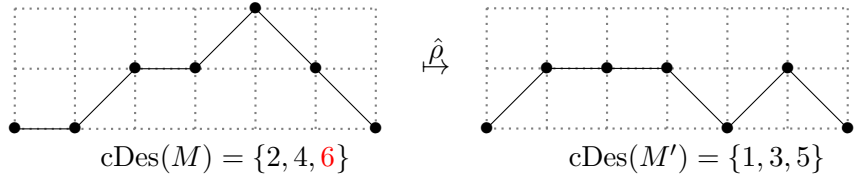
\begin{figure}
	 \begin{center}
	 		\begin{tikzpicture}[scale=0.8]
 		\draw[step=1cm, gray, thick,dotted] (0, 0) grid (6,2);
 		
 		\draw[black] (0, 0)--(1, 0)--(2, 1)--(3, 1)--(4, 2)--(5, 1)--(6, 0);
 		\draw[black] (0,0) node {$\bullet$};
 		\draw[black] (1,0) node {$\bullet$};
 		\draw[black] (2,1) node {$\bullet$};
 		\draw[black] (3,1) node {$\bullet$};
 		\draw[black] (4,2) node {$\bullet$};
 		\draw[black] (5,1) node {$\bullet$};
 		\draw[black] (6,0) node {$\bullet$};
                 \draw[black] (3.5,-0.5) node {$\cDes(M)=\{2,4,\red{6}\}$};
                \draw[black] (7.0, 1) node {$\mapsto$};
              \draw[black](7.0, 1.3) node {$\hat\rho$};
 		\end{tikzpicture}
		\begin{tikzpicture}[scale=0.8]
 		\draw[step=1cm, gray, thick,dotted] (0, 0) grid (6,2);
 		\draw[black] (0, 0)--(1, 1)--(2, 1)--(3, 1)--(4, 0)--(5, 1)--(6, 0);
 		\draw[black] (0,0) node {$\bullet$};
 		\draw[black] (1,1) node {$\bullet$};
 		\draw[black] (2,1) node {$\bullet$};
 		\draw[black] (3,1) node {$\bullet$};
 		\draw[black] (4,0) node {$\bullet$};
 		\draw[black] (5,1) node {$\bullet$};
 		\draw[black] (6,0) node {$\bullet$};
                 \draw[black] (3.5,-0.5) node {$\cDes(M')=\{1,3, 5\}$};
 		\end{tikzpicture}
		\caption{Case\,3(b):\,\,\,\,\,\,$\hat\rho(M)=M'$}\label{2edcase3a}		 
	 \end{center}
	 \end{figure}	 

\end{enumerate}

\end{enumerate}
\end{definition}
%


  \begin{lemma}\cite[Observation~1.4]{AER}\label{isomop}
 Let $T_1$ and $T_2$ be finite sets, let $\theta:T_2\mapsto T_1$ be a bijection, and let $\Stat_1:T_1\mapsto 2^{[n-1]}$ and $\Stat_2:T_2\mapsto 2^{[n-1]}$ be maps such that $\Stat_2=\Stat_1\circ \theta$. If $(\cStat_1, \phi_1)$ is a cyclic extension of $\Stat_1$, then $(\cStat_2, \phi_2)$ with $\cStat_2=:\cStat_1\circ\theta$ and $\phi_2:=\theta^{-1}\circ\phi_1\circ\theta$ is a cyclic extension of $\Stat_2$.
 
 \end{lemma}

Using the bijection $\varphi$ in Lemma~\ref{bijection:Mot}, it is easy to check that 
$\cDes=\widehat{\cDes}\circ \varphi^{-1}$ and $\rho=\varphi\circ \hat\rho\circ \varphi^{-1}$. By Lemma~\ref{isomop} and Corollary~\ref{coro:cde}, we obtain the following equivalent cyclic extension.

\begin{theorem}\label{2edcdesmotz}
 The pair $(\widehat\cDes, \hat\rho)$, with ${\widehat\cDes}$ and $\hat\rho$ determined by Definition~\ref{def2:Mcdes} and Definition~\ref{def:rho1}, respectively, is a cyclic extension of ${\widehat\Des}$ on $\M^*_n$.
\end{theorem}



\section{Cyclic descents of Motzkin paths in terms of standard Young
tableaux }\label{sec:SYT}

Another important family of combinatorial objects, called standard Young tableaux (SYT), have a well-studied notion of descent set.
Recall that the Young diagram corresponding to a partition $\lambda=(\lambda_1,\ldots, \lambda_k)$ with $\lambda_1\geq\cdots\geq\lambda_k$ consists of left-justified rows so that the $i$-th ($1\leq i\leq k$) row from the top contains $\lambda_i$ cells.
Let $\lambda/\mu$ denote a skew shape, where $\lambda$ and $\mu$ are partitions such that
the Young diagram of $\mu$ is contained in that of $\lambda$,
see Figure~\ref{fig:SYT} for an example.

The {\em standard Young tableaux} of skew shape $\lambda/\mu$ are the labelings of
the cells of the diagram $\lambda/\mu$ with positive integers from $1$ to the number of cells,
such that each number appears once, row entries are increasing from north to south,
and column entries are increasing from west to east.
Let $\SYT(\lambda/\mu)$ denote the set of standard Young tableaux of
shape $\lambda/\mu$. For a straight shape $\lambda=(\lambda_1,\lambda_2,\dots)$, we write $\SYT(\lambda_1,\lambda_2,\dots)$ instead of $\SYT((\lambda_1,\lambda_2,\dots))$ for simplicity.
The {\em descent set} of $T \in \SYT(\lambda/\mu)$ is
\[
\Des(T) := \{i\in[n-1] \,:\, i+1 \text{ is in a lower row than $i$ in $T$}\},
\]
where $n$ is the size of skew shape. 
For example, the descent set of the SYT in Figure~\ref{fig:SYT} is $\{1,4,7,8\}$.

\begin{figure}[htb]
$$\young(::147,2368,59)$$
\caption{A SYT of shape $(5,4,2)/(2)$.} 
\label{fig:SYT}
\end{figure}

%

We introduce the Sch\"utzenberger promotion operator.
\begin{definition}\label{def:pro}
    The action of the {\em promotion} operator
    on $T\in \SYT(\lambda/\mu)$ is as follows.
    Add $1 \pmod n$ to each entry of $T$;
    this turns the $n$ into a $1$ and increments the remaining entries.
    Repeatedly swap $1$ with the larger of its northern neighbor (if it exists)
    and its western neighbor (if it exists), until the cell containing $1$
    has neither northern nor western neighbor.

\end{definition}
It is not difficult to show that the resulting tableau is standard.
Let $\pro$ be the (Sch\"utzenberger) promotion operator\footnote{In our notation, the definition of promotion is consistent with that in the literature~\cite{Huang, Rhoades}.}. 

 Recall that the skew shape $(n-k, n-k, n-2k)/(n-2k, n-2k)\, (0<k\leq \lfloor n/2\rfloor)$ has the following form 
  \begin{center}
\ytableausetup
{mathmode, boxsize=1.5em}
\begin{ytableau}
\none & \none & \none & \none & \none & \space & \space & \space & \none[\dots] & \space & \space\\
\none & \none & \none & \none & \none & \space & \space &  \space & \none[\dots] & \space & \space\\
\space & \space  & \none[\dots] & \space &  \space & \none &\none &\none &\none &\none
\end{ytableau}. 
 \end{center}

\begin{lemma}\label{bijection:motzsyt}
There exists a bijection 
$$\Gamma: \M_n\mapsto \sum_{0\leq k\leq \lfloor n/2\rfloor}\SYT(n-k, n-k, n-2k)/(n-2k, n-2k)$$ 
with the property that  
${\Des}(M)=\Des(\Gamma(M))\,\,\,\, {\forall}\,\,\,\, M\in\M_n.$
\end{lemma}

\begin{proof}
 
 For $M\in\M_n$, let $k$ be the number of $\su$ steps (equivalently, $\sd$ steps) in $M$. Define $\Gamma\,(M)$ by  inserting the entries from $1$ to $n$ into the shape $\SYT((n-k, n-k, n-2k)/(n-2k, n-2k))$ as follows:
insert $i$ in the first (resp. second, third) row if the $i$th step is a $\su$ step (resp. $\sd$ step, $\sl$ step),
 It is easy to see that $\Gamma$ is a bijection. 
 By Definition~\ref{def:motzdes} $i\in{\Des}(M)$ if and only if $i+1$ is in a lower row than $i$ in $\Gamma(M)$, namely $i\in\Des(\Gamma(M))$. 
 This completes the proof.
 \end{proof}
 
 An important concept is the {\it jeu de taquin} (jdt) construction; see \cite[Appendix A1.2]{EC2} for a detailed description.
For any SYT $T$ of skew shape, denote by $\Jdt(T)$
the SYT of straight shape obtained from $T$ by performing a sequence of jdt slides.
 Jeu de taquin has the following well-known property.

\begin{lemma}{\cite[p.431]{EC2}}\label{lem:jdtDes}
For any SYT $T$ of skew shape,
\[
\Des(\Jdt(T))=\Des(T).
\]
\end{lemma} 
 
 Combining 
 Lemma~\ref{bijection:motzsyt} and Lemma~\ref{lem:jdtDes}, we obtain the following theorem.
 \begin{proposition}\label{prop:Mdes}
There exists a bijection $\tilde{\Gamma}:= \Jdt\circ \Gamma$, which maps the set of Motzkin paths  $\M_n$ to the set of SYT of straight shape of size $n$ and at most $3$ rows, 
with the property that  
$$\Des(M)=\Des(\tilde\Gamma(M))\,\,\,\, {\forall}\,\,\,\, M\in\M_n.$$

 \end{proposition}
 
 For the $\SYT$ of shape $(n-k, n-k, n-2k)/(n-2k, n-2k)$, define the cyclic descent map as follows. 
 
 \begin{definition}\label{def:cDes3row}
For $T\in \SYT((n-k, n-k, n-2k)/(n-2k, n-2k))\, (0<k\leq \lfloor n/2\rfloor)$, let $n\in\cDes(T)$ if and only if either of the following conditions
holds:
\begin{enumerate}
\item $1$ and $n$ are in the third row and second row of $T$, respectively.
 \item $1$ and $n$ are in the first row and second row of $T$, respectively, and for every $1<i<k$,  $T_{2,i-1} > T_{1,i}$, where $T_{i, j}$ denotes the entry in row $i$ (from the top) and column $j$ (from the left) of $T$.
\end{enumerate}
\end{definition}

\begin{example} According to Definition~\ref{def:cDes3row}, $6\in\cDes\left(\young(::25,::46,13)\right)$ and $6\in\cDes\left(\young(::13,::46,25)\right)$.
\end{example}

 \begin{proposition}\label{prop:comm}
For $n\geq 1$, we have

\begin{equation}\label{eq:commu}
\pro\circ\Gamma=\Gamma\circ{\rho}.
\end{equation}

\end{proposition}  

\begin{proof}
For $M\in\M_n$, 
combining the definition of $\Gamma$ in the proof of Lemma~\ref{bijection:motzsyt} and that of $\pro$ in Definition~\ref{def:pro}, it is easy to check that Eq. \eqref{eq:commu} holds for all the cases in Definition~\ref{def:motzrho}. For example, in case $1$,  for $M=\sp\sl$ and ${\rho}(M)=\sl\sp$, the entries $T_{i, j}+1$ of $\Gamma({\rho}(M))$ from the segment $\sp$ lie in the $j$th ($1\leq j\leq 3$) row if the enries $T_{i,j}$ of $\Gamma(M)$ from the segment $\sp$ lie in the $j$th row. 
Since the entry $n$ of $\Gamma(M)$ lies in the third row and the entry $1$ of $\Gamma(\rho(M))$ lies in the third row,  it is easy to check that the $\pro\circ\Gamma(T)$ coincides with $\Gamma({\rho}(M))$.

\end{proof}
 
Our cyclic extension for $\SYT$ of shape $(n-k, n-k, n-2k)/(n-2k, n-2k)$ is stated as follows.
 
\begin{theorem}\label{the:3rows}
For each $0<k\leq \lfloor n/2\rfloor$, $\Des$ on $\SYT((n-k, n-k, n-2k)/(n-2k, n-2k))$ has a cyclic extension $(\cDes, \phi)$. The map $\cDes$ is defined by Definition~\ref{def:cDes3row}. The bijection $\phi:\SYT((n-k, n-k, n-2k)/(n-2k, n-2k))\mapsto\SYT((n-k, n-k, n-2k)/(n-2k, n-2k))$ is defined by $\phi:=\pro$.
	\end{theorem}
 \begin{proof}
 By Definition~\ref{def:Mcdes} and Definition~\ref{def:cDes3row}, it is easy to check that the bijection $\Gamma$ satisfies the property 
 $$\cDes(\Gamma(M))={\cDes}(M)\,\,\, \quad\textit{for}\,\,\,\quad M\in\M_n.$$  
Using Proposition~\ref{prop:comm}, the map $\phi$  
 is the composition of bijections $\Gamma\circ {\rho} \circ \Gamma^{-1}$ described directly in terms of $\SYT$.
 By Lemma~\ref{isomop} and Lemma~\ref{bijection:motzsyt}, we obtain the cyclic descent extension $(\cDes, \phi)$ of $\Des$ on $\SYT((n-k, n-k, n-2k)/(n-2k, n-2k))$. 
 \end{proof}
 
 \begin{remark}
Recall \cite[Corollary~6.4]{Huang}, the cyclic extension $(\cDes, \phi)$ in Theorem~\ref{the:3rows} coincides with the cyclic descent extension in \cite[Theorem~1.9]{Huang}.
 \end{remark}

 \begin{example}
 Below are two of the orbits of the $\ZZ$-actions on \SYT((6-k, 6-k, 6-2k)/(6-2k, 6-2k)) for $k=1,2$  generated by $\phi$, respectively. 
 
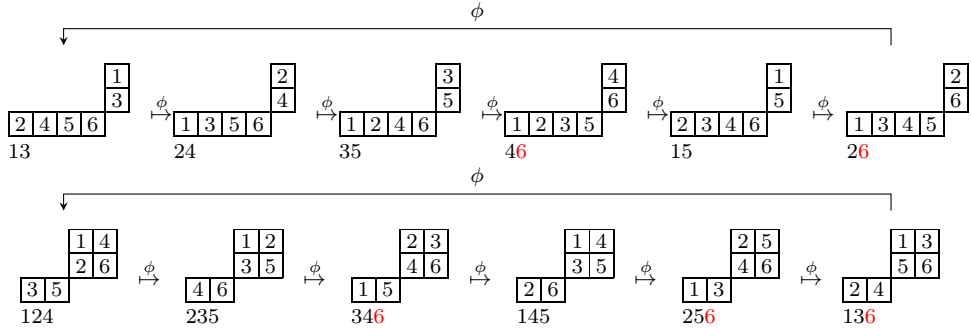
\begin{tikzpicture}[scale=2.2]\scriptsize

\node (0) at (0,11) {$\younganddescent{::::1,::::3,2456}{13} \phimapsto$};
\node (1) at (1,11) {$\younganddescent{::::2,::::4,1356}{24} \phimapsto$};
\node (2) at (2,11) {$\younganddescent{::::3,::::5,1246}{35} \phimapsto$};
\node (3) at (3,11) {$\younganddescent{::::4,::::6,1235}{4\red{6}} \phimapsto$};
\node (4) at (4,11) {$\younganddescent{::::1,::::5,2346}{15}\phimapsto$};
\node (5) at (5.0,11) {$\younganddescent{::::2,::::6,1345}{2\red{6}}$};
\draw [->, >=stealth] (4.9,11.42) -- (4.9,11.52) -- node[above]{\scriptsize $\phi$} (-0.1,11.52) -- (-0.1,11.42);

\node (6) at (0,10) {$\younganddescent{::14,::26,35}{124} \phimapsto$};
\node (7) at (1,10) {$\younganddescent{::12,::35,46}{235} \phimapsto$};
\node (8) at (2,10) {$\younganddescent{::23,::46,15}{34\red{6}} \phimapsto$};
\node (9) at (3,10) {$\younganddescent{::14,::35,26}{145} \phimapsto$};
\node (10) at (4,10) {$\younganddescent{::25,::46,13}{25\red{6}}\phimapsto$};
\node (11) at (4.9,10) {$\younganddescent{::13,::56,24}{13\red{6}}$};
\draw [->, >=stealth] (4.9,10.42) -- (4.9,10.52) -- node[above]{\scriptsize $\phi$} (-0.1,10.52) -- (-0.1,10.42);

\end{tikzpicture}
\end{example}

\section*{Acknowledgement}

The author thanks Prof.~Adin and Prof.~Roichman for their helpful suggestions. The author also thanks the anonymous referees for their valuable comments.


\end{document}